\documentclass[11pt]{article}

\usepackage[T1]{fontenc}
\usepackage[margin=1.02in]{geometry}
\usepackage{amsmath,amssymb,amsthm,mathtools}
\usepackage{array,enumitem,microtype}
\usepackage[dvipsnames]{xcolor}
\usepackage[
  colorlinks=true,
  linkcolor=MidnightBlue,
  citecolor=MidnightBlue,
  urlcolor=MidnightBlue,
  pdftitle={Unit Actions on Homometric Five-Point Subsets of Cyclic Groups},
  pdfauthor={Pakin Methawisal},
  pdfsubject={Unit-group actions on homometric five-point subsets},
  pdfkeywords={homometric sets, binary bracelets, cyclic groups, unit-group actions, periodic phase retrieval, Burnside lemma},
]{hyperref}
\usepackage[nameinlink,capitalize,noabbrev]{cleveref}

\newtheorem{theorem}{Theorem}[section]
\newtheorem{proposition}[theorem]{Proposition}
\newtheorem{lemma}[theorem]{Lemma}
\newtheorem{corollary}[theorem]{Corollary}

\newcommand{\ord}{\operatorname{ord}}
\newcommand{\Fix}{\operatorname{Fix}}
\newcommand{\Acal}{\mathcal A}
\newcommand{\Bcal}{\mathcal B}
\newcommand{\Ccal}{\mathcal C}
\newcommand{\Dcal}{\mathcal D}
\newcommand{\Ecal}{\mathcal E}
\newcommand{\Fcal}{\mathcal F}
\newcommand{\Gcal}{\mathcal G}

\title{Unit Actions on Homometric Five-Point Subsets of Cyclic Groups\\
\large Orbit and fixed-point refinement of the seven-family classification}
\author{
Pakin Methawisal\\
Mahidol University International College, Mahidol University\\
\texttt{pakin.met@student.mahidol.edu}
}
\date{}

\begin{document}
\maketitle

\begin{abstract}
Erickson and Jones classified the nontrivial homometry classes of five-point binary bracelets into seven families, Types~A--G, and conjectured that each family is invariant under multiplication by units modulo sign. We prove this conjecture and show that the induced actions are diagonal for Type~A, multiplicative for Types~B--E, affine for Type~F, and regular through $U_{20}/\{\pm1\}\cong C_4$ for Type~G. We derive closed arithmetic formulas for the numbers of unit orbits and classes fixed by every unit, and determine the orbit-size distributions in all seven types, including explicit stabilizer descriptions for Types~A and~F.
\end{abstract}



\section{Introduction}

Recovering a periodic binary signal from its Fourier magnitude is a basic phase-retrieval problem \cite{BendoryEdidin2020,Rosenblatt1984}. The magnitude determines the periodic autocorrelation, but it need not determine the support: distinct subsets of a cyclic group can have the same multiset of pairwise differences modulo sign, equivalently the same pairwise Lee-distance multiset. Such subsets are called homometric. In mathematical music theory, the same relation is known as the Z-relation \cite{Mandereau2011a,Mandereau2011b}.

Erickson and Jones classified the nontrivial homometry classes of five-point binary bracelets in $\mathbb Z_n$ into seven disjoint parametric types A--G \cite{EricksonJones2026}. They also asked for the corresponding unit-orbit structure and conjectured that each type is invariant under multiplication by units modulo sign \cite[Problem~5.3 and Conjecture~5.4]{EricksonJones2026}.

We answer these questions by replacing the original indices with intrinsic sign-class parameters. This makes the unit actions explicit and reduces the resulting orbit and fixed-point problems to elementary arithmetic and Burnside's lemma. The resulting formulas give a unit-orbit refinement of the seven-family classification. We further determine the orbit-size distributions in all seven types, with explicit stabilizer descriptions for the diagonal and affine actions in Types~A and~F.

\section{Families and intrinsic parameters}\label{sec:families}

Let $\mathbb Z_N=\mathbb Z/N\mathbb Z$. For a subset $S\subseteq\mathbb Z_N$, write $T_c(S)=S+c$ and $R_c(S)=c-S$ for translation by $c$ and the reflection $x\mapsto c-x$, respectively. A five-point binary bracelet is a dihedral orbit $[S]$ of a subset $S\subseteq\mathbb Z_N$, where the dihedral action is generated by the maps $T_c$ and $R_c$. For brevity, if \(S=\{a_1,\ldots,a_5\}\), we write \([a_1,\ldots,a_5]\) for the bracelet \([S]\).

For $a,b\in\mathbb Z_N$, choose representatives in $\{0,\ldots,N-1\}$. Their Lee distance is
\[
\min\{|a-b|,\,N-|a-b|\}.
\]
Two bracelets are homometric when their multisets of pairwise Lee distances agree.

Write
\[
 U_N=(\mathbb Z_N)^\times,
 \qquad
 \widetilde U_N=U_N/\{\pm1\},
 \qquad
 Q_N=\mathbb Z_N/\{x\sim -x\}.
\]
The sign class of $x$ is denoted $[x]_\pm$, and its order is the additive order of either representative. For $x\in\mathbb Z_R$,
\[
\ord_R(x)=\frac{R}{\gcd(R,x)}.
\]

\begin{lemma}\label{lem:homometry}
Multiplication by units preserves homometry and therefore induces an action of $\widetilde U_N$ on homometry classes.
\end{lemma}

\begin{proof}
First, multiplication by a unit is well defined on bracelets. For $u\in U_N$,
\[
uT_c(S)=u(S+c)=uS+uc=T_{uc}(uS),
\]
and
\[
uR_c(S)=u(c-S)=uc-uS=R_{uc}(uS).
\]
Thus multiplication by $u$ carries translations to translations and reflections to reflections, and hence sends dihedrally equivalent subsets to dihedrally equivalent subsets.

Two differences $x,y\in\mathbb Z_N$ have the same Lee distance exactly when $[x]_\pm=[y]_\pm$. Thus two bracelets are homometric exactly when their multisets of difference sign classes agree. Multiplication by a unit induces the permutation
\[
[x]_\pm\longmapsto[ux]_\pm
\]
of $Q_N$, so it applies the same permutation to both difference-sign multisets and therefore preserves homometry.

Finally, $-S=R_0(S)$, so multiplication by $-1$ acts trivially on bracelets. Hence $u$ and $-u$ induce the same action, and the $U_N$-action factors through $\widetilde U_N=U_N/\{\pm1\}$.
\end{proof}

The unit action and its preservation of homometry are also described in \cite[Section~5]{EricksonJones2026}. We shall use the standard fact that if $R\mid N$, then the reduction map $U_N\to U_R$ is surjective.

For the seven families, let $\Acal_m(i,j),\ldots,\Gcal_m(i)$ denote entire homometry classes. We use the representatives obtained by scaling the continuous representatives in \cite[Proposition~3.2]{EricksonJones2026}, using the unified Type~G parametrization in \cite[Remark~3.3]{EricksonJones2026}. Their representatives are
\begin{align*}
\Acal_m(i,j):\ &N=2m,
 &&\big\{[0,i,i+j,m+i-j,m],\
 [0,m+i,i+j,m+i-j,m]\big\};\\
\Bcal_m(i):\ &N=5m,
 &&\big\{[0,i,m,2m,2m+i],\
 [0,i,m,m+i,3m]\big\};\\
\Ccal_m(i):\ &N=6m,
 &&\big\{[0,m+i,2m,2m+i,4m],\
 [0,2m,2m+i,3m+i,4m]\big\};\\
\Dcal_m(i):\ &N=6m,
 &&\big\{[i,m-i,2m,2m+i,3m],\
 [0,i,m-i,2m+i,5m]\big\};\\
\Ecal_m(i):\ &N=6m,
 &&\big\{[0,i,m+i,2m+i,3m],\
 [0,i,m,2m,3m+i], [0,m-i,m,2m,4m-i]\big\};\\
\Fcal_m(i):\ &N=8m,
 &&\big\{[0,i,m,2m+i,4m],\
 [0,i,m+i,2m,4m+i]\big\};\\
\Gcal_m(i):\ &N=20m,
 &&\big\{[0,5m,8mi,15m+4mi,5m-4mi], [0,5m,4mi,15m+8mi,15m-4mi]\big\}, \\
 &&&\qquad i\in\{1,2,3,4\}.
\end{align*}

The original index ranges become particularly simple after passing to sign classes.

Explicitly, the parameter maps are
\[
\Acal_m(i,j)\longmapsto([i]_\pm,[j]_\pm),\qquad
\Bcal_m(i),\ldots,\Fcal_m(i)\longmapsto[i]_\pm,\qquad
\Gcal_m(i)\longmapsto i,
\]
where the ambient sign-class space is the one specified below.

For Type~F, write
\[
\langle m\rangle=\{0,m,\ldots,7m\}\leq\mathbb Z_{8m}.
\]

\begin{proposition}[Intrinsic parameter spaces]\label{prop:parameters}
The seven families are parameterized by the following sets:
\begin{enumerate}[label=\textup{(\Alph*)},leftmargin=2.35em]
\item ordered pairs $(\alpha,\beta)$ of nonzero sign classes in $Q_m$, neither of order $2$, such that
\[
\alpha\ne\beta,
\qquad
\ord(\beta)=3\Longrightarrow\ord(\alpha)=6;
\]
\item nonzero sign classes in $Q_{5m}$ excluding orders $5$ and $10$;
\item all nonzero sign classes in $Q_m$;
\item nonzero sign classes in $Q_{3m}$ excluding orders $2$, $3$, and $6$;
\item nonzero sign classes in $Q_m$ excluding order $2$;
\item sign classes in $Q_{8m}$ represented outside $\langle m\rangle$;
\item the four labels $\{1,2,3,4\}$.
\end{enumerate}
\end{proposition}

\begin{proof}
The following parameter identities hold:
\begin{align*}
\Acal_m(i+m,j)&=\Acal_m(i,j),&
\Acal_m(i,j+m)&=\Acal_m(i,j),\\
\Acal_m(-i,j)&=\Acal_m(i,j),&
\Acal_m(i,-j)&=\Acal_m(i,j),\\
\Bcal_m(-i)&=\Bcal_m(i),&
\Bcal_m(i+5m)&=\Bcal_m(i),\\
\Ccal_m(-i)&=\Ccal_m(i),&
\Ccal_m(i+m)&=\Ccal_m(i),\\
\Dcal_m(-i)&=\Dcal_m(i),&
\Dcal_m(i+3m)&=\Dcal_m(i),\\
\Ecal_m(-i)&=\Ecal_m(i),&
\Ecal_m(i+m)&=\Ecal_m(i),\\
\Fcal_m(-i)&=\Fcal_m(i),&
\Fcal_m(i+8m)&=\Fcal_m(i).
\end{align*}
They are verified by explicit translations and reflections in \cref{app:param-identities}.

Choosing the least nonnegative representative of each sign class gives the following exact correspondence; exceptional values such as \(m/2\) are omitted whenever they are nonintegral:
\[
\begin{array}{c|c|c}
\text{Type} & \text{canonical indices of \cite[Theorem~4.1]{EricksonJones2026}} & \text{intrinsic description} \\
\hline
\mathrm A &
\text{\small$\begin{matrix} 1\le i,j<m/2,\ i\ne j, \\ j=m/3\Longrightarrow i=m/6 \end{matrix}$} &
\text{\small$\begin{matrix} \alpha,\beta\ne0,\ \ord(\alpha),\ord(\beta)\ne2,\ \alpha\ne\beta, \\ \ord(\beta)=3\Longrightarrow\ord(\alpha)=6 \end{matrix}$} \\
\mathrm B & 1\le i\le5m/2,\ i\notin\{m/2,m,3m/2,2m\} & \alpha\ne0,\ \ord(\alpha)\notin\{5,10\} \\
\mathrm C & 1\le i\le m/2 & \alpha\ne0\text{ in }Q_m \\
\mathrm D & 1\le i<3m/2,\ i\notin\{m/2,m\} & \alpha\ne0,\ \ord(\alpha)\notin\{2,3,6\} \\
\mathrm E & 1\le i<m/2 & \alpha\ne0,\ \ord(\alpha)\ne2 \\
\mathrm F & 1\le i<4m,\ i\notin\{m,2m,3m\} & [i]_\pm\not\subseteq \langle m\rangle.
\end{array}
\]
Using the formula for $\ord_R(x)$, the exclusions in the canonical index sets correspond exactly to the excluded order strata in the intrinsic descriptions: orders $5$ and $10$ in Type~B, orders $2,3,6$ in Type~D, order $2$ in Type~E, and the sign classes contained in $\langle m\rangle$ in Type~F. Hence the intrinsic parameter sets above are in bijection with the canonical index sets of \cite[Theorem~4.1]{EricksonJones2026}.

For each Type~$T\in\{\mathrm A,\ldots,\mathrm F\}$, the proof of \cite[Lemma~4.3]{EricksonJones2026} shows that the number of Type~$T$ homometry classes equals the cardinality of its canonical index set. The classification theorem already gives a surjection from that index set onto the Type~$T$ classes, so this map is bijective. This proves the asserted parameterizations for Types~A--F. The distance calculation in the proof of \cref{thm:action} distinguishes the four Type~G classes.
\end{proof}

\section{Proof of the unit-invariance conjecture}\label{sec:action}

For Type~F, define
\[
\varepsilon(u)=
\begin{cases}
0,&u\equiv\pm1\pmod8,\\
1,&u\equiv\pm3\pmod8,
\end{cases}
\qquad u\in U_{8m},
\]
and define a map on sign classes by
\[
\Phi_u([i]_\pm)=[ui+4m\varepsilon(u)]_\pm.
\]
This is independent of the representative of the sign class: replacing \(i\) by \(-i\) changes the displayed representative to
\[
-ui+4m\varepsilon(u)\equiv-(ui+4m\varepsilon(u))\pmod{8m}.
\]
It is also independent of the choice of representative of \(u\) modulo sign, since \(\varepsilon(-u)=\varepsilon(u)\) and
\[
[-ui+4m\varepsilon(u)]_\pm=[ui+4m\varepsilon(u)]_\pm.
\]
Thus \(\Phi_u\) depends only on the class of \(u\) in \(\widetilde U_{8m}\).

\begin{theorem}[Unit action on Types A--G]\label{thm:action}
Under the parameterizations of \cref{prop:parameters}, multiplication by $u\in U_N$ acts as
\begin{align*}
\textup{A: }&(\alpha,\beta)\longmapsto(u\alpha,u\beta),\\
\textup{B--E: }&\alpha\longmapsto u\alpha,\\
\textup{F: }&\alpha\longmapsto\Phi_u(\alpha),\\
\textup{G: }&\text{a regular action through }U_{20}/\{\pm1\}\cong C_4.
\end{align*}
Consequently every Type~A--G is invariant under $\widetilde U_N$.
\end{theorem}

\begin{proof}
\textbf{Type A.}
Here $N=2m$. Every unit is odd, hence $um\equiv m\pmod{2m}$, and direct substitution gives
\[
 u\Acal_m(i,j)=\Acal_m(ui,uj).
\]

\smallskip
\noindent\textbf{Type B.}
Modulo sign, choose $u\equiv1$ or $2\pmod5$ and put $t=ui$. The first case gives $u\Bcal_m(i)=\Bcal_m(t)$. In the second case the image bracelets are
\[
X=[0,t,2m,4m,4m+t],
\qquad
Y=[0,t,m,2m,2m+t].
\]
Writing the two bracelets of $\Bcal_m(-t)$ as $B^{(1)}_{-t}$ and $B^{(2)}_{-t}$ in the order displayed in \cref{sec:families}, one has in $\mathbb Z_{5m}$
\[
R_0(X)=B^{(2)}_{-t},
\qquad
R_{2m}(Y)=B^{(1)}_{-t}.
\]
Hence
\[
u\Bcal_m(i)=\Bcal_m(-t)=\Bcal_m(t).
\]

\smallskip
\noindent\textbf{Types C, D, and E.}
Here $N=6m$. Since $U_6=\{\pm1\}$, after passing modulo sign we may choose a representative with $u\equiv1\pmod6$. Hence, for every integer $r$,
\[
u(rm)\equiv rm\pmod{6m}.
\]
All fixed multiples of $m$ in the displayed representatives are therefore unchanged, while each occurrence of $i$ is replaced by $ui$. Direct substitution gives
\[
 u\Ccal_m(i)=\Ccal_m(ui),
 \qquad
 u\Dcal_m(i)=\Dcal_m(ui),
 \qquad
 u\Ecal_m(i)=\Ecal_m(ui).
\]

\smallskip
\noindent\textbf{Type F.}
Modulo sign, choose $u\equiv1$ or $3\pmod8$ and put $t=ui$. The first case gives $u\Fcal_m(i)=\Fcal_m(t)$. If $u\equiv3\pmod8$, the image bracelets are
\[
X'=[0,t,3m,6m+t,4m],
\qquad
Y'=[0,t,3m+t,6m,4m+t].
\]
Set $s=4m-t$. Writing the two bracelets of $\Fcal_m(s)$ as $F^{(1)}_s$ and $F^{(2)}_s$ in the order displayed in \cref{sec:families}, one has in $\mathbb Z_{8m}$
\[
R_{4m}(X')=F^{(1)}_s,
\qquad
R_0(Y')=F^{(2)}_s.
\]
Hence
\[
 u\Fcal_m(i)=\Fcal_m(s)=\Fcal_m(4m-ui).
\]
Since
\[
[4m-ui]_\pm=[ui+4m]_\pm,
\]
this is exactly the action \(\Phi_u\) defined above.

\smallskip
\noindent\textbf{Type G.}
Every coordinate in the Type~G representatives is a multiple of $m$, so the action depends only on $u\bmod20$. Dividing by $m$ and normalizing dihedrally gives the four classes
\begin{align*}
1&:\ \{[0,1,2,5,14],[0,1,2,6,9]\},\\
2&:\ \{[0,1,3,8,17],[0,1,4,7,9]\},\\
3&:\ \{[0,1,3,9,16],[0,1,4,6,13]\},\\
4&:\ \{[0,1,3,7,12],[0,1,4,13,15]\}.
\end{align*}
For $i=1,2,3,4$, the common normalized distance multiset contains each of $1,\ldots,9$ once, with the repeated distance respectively $1,3,7,$ and $9$. Hence the four homometry classes are distinct.

Write the two normalized bracelets in row $i$ as $G_i^{(1)}$ and $G_i^{(2)}$. The four transitions under multiplication by $3$ are witnessed explicitly by
\[
\begin{array}{c|cc}
1\to2&R_3(3G_1^{(1)})=G_2^{(1)}&R_7(3G_1^{(2)})=G_2^{(2)}\\
2\to4&R_4(3G_2^{(1)})=G_4^{(2)}&T_0(3G_2^{(2)})=G_4^{(1)}\\
4\to3&T_0(3G_4^{(1)})=G_3^{(1)}&T_1(3G_4^{(2)})=G_3^{(2)}\\
3\to1&R_9(3G_3^{(1)})=G_1^{(2)}&T_2(3G_3^{(2)})=G_1^{(1)}.
\end{array}
\]
Thus multiplication by $3$ gives the four-cycle
\[
1\longmapsto2\longmapsto4\longmapsto3\longmapsto1.
\]
The quotient $U_{20}/\{\pm1\}$ has four elements, and $[3]$ has order $4$, so it generates this quotient. Since the generator acts transitively as a four-cycle on four classes, the action of $U_{20}/\{\pm1\}$ is free and transitive, hence regular. Finally, surjectivity of $U_{20m}\to U_{20}$ shows that the action induced from $U_{20m}$ is the same regular action.

Finally, multiplication by a unit is bijective and preserves additive order. In Type~F, the map \(\Phi_u\) carries \(\langle m\rangle\) to itself, and therefore also preserves its complement. Hence every admissible parameter remains admissible.
\end{proof}

\begin{corollary}\label{cor:conjecture}
\Cref{thm:action} proves Conjecture~5.4 of \cite{EricksonJones2026}. Together with the seven-family classification, it decomposes all nontrivial five-point homometry classes into unit orbits.
\end{corollary}

\section{Orbit counts and globally fixed classes}\label{sec:orbits}

The actions in \cref{thm:action} lead to four orbit-counting mechanisms. For Types~B--E, additive order completely determines the orbit. Type~A is a diagonal action on ordered pairs and is treated by fixed-point counting and Burnside's lemma, while Type~F is affine and is resolved by the orbit decomposition in \cref{thm:F-distribution}. Type~G is regular.

If a homometry class \(\mathcal H\) is fixed by every unit, meaning $u\mathcal H=\mathcal H$ for every $u\in U_N$, we call it \emph{globally fixed}. Notice that this condition concerns the homometry class as a whole; its constituent bracelets need not be fixed individually.

If a finite group $G$ acts on a finite set $X$, we use Burnside's lemma in the form
\[
\#(X/G)=\frac1{|G|}\sum_{g\in G}|\Fix_X(g)|,
\qquad
\Fix_X(g)=\{x\in X:g\cdot x=x\}.
\]

We begin with the order-theoretic fact used for Types~B--E.

\begin{lemma}\label{lem:orders}
For every \(q\mid N\), the group \(U_N\) is transitive on the elements of additive order \(q\) in \(\mathbb Z_N\), and \(\widetilde U_N\) is transitive on their sign classes. A nonzero order-\(q\) sign class is fixed by every unit exactly when
\[
q\in\{2,3,4,6\}.
\]
\end{lemma}

\begin{proof}
Every order-\(q\) element of \(\mathbb Z_N\) has the form \((N/q)a\) with \(a\in U_q\). Since every unit modulo \(q\) lifts to one modulo \(N\), the unit group is transitive on the order-\(q\) elements, and hence also on their sign classes.

There are $\varphi(q)$ order-$q$ elements, where $\varphi$ denotes Euler's totient function. If \(q>2\), then \(x\ne-x\), so each sign class has two elements and the unique orbit of order-\(q\) sign classes has size \(\varphi(q)/2\). If \(q=2\), the unique order-two element \(N/2\) is self-negative, so there is a single sign class.

Thus an order-\(q\) sign class is fixed by every unit exactly when this orbit has size one: either \(q=2\), or \(\varphi(q)=2\). The latter occurs exactly for \(q\in\{3,4,6\}\).
\end{proof}

For a positive integer \(r\), write
\[
\delta_d(r)=\mathbf 1_{d\mid r},
\qquad
\tau(r)=\#\{d:d\mid r\},
\qquad
\psi(r)=r\prod_{p\mid r}\left(1+\frac1p\right),
\]
where \(\psi\) is the Dedekind psi function, and set
\[
P(r)=\sum_{d\mid r}\psi(d).
\]
We also use the multiplicative function \(\kappa\) determined by
\[
\kappa(1)=1,
\qquad
\kappa(p^a)=
\begin{cases}
4a,&p=2,\ a\ge1,\\
2a+1,&p\text{ odd}.
\end{cases}
\]

\begin{theorem}[Orbit counts and globally fixed classes]\label{thm:counts}
Let \(O_T(m)\) be the number of unit orbits and \(F_T(m)\) the number of globally fixed classes in Type~\(T\). Then
\[
\begin{array}{c|c|c}
\text{Type} & O_T(m) & F_T(m)\\
\hline
\mathrm A &
\displaystyle
\begin{aligned}
\frac{P(m)+\kappa(m)}2
&-\bigl(3+2\delta_2(m)+\delta_3(m)\bigr)\tau(m)\\
&+2+4\delta_2(m)+2\delta_3(m)+2\delta_6(m)
\end{aligned}
&
\displaystyle 2\delta_6(m)+3\delta_{12}(m)
\\[1.4em]
\mathrm B &
\tau(5m)-2-\delta_2(m)
&
\delta_2(m)+\delta_3(m)+\delta_4(m)+\delta_6(m)
\\
\mathrm C &
\tau(m)-1
&
\delta_2(m)+\delta_3(m)+\delta_4(m)+\delta_6(m)
\\
\mathrm D &
\tau(3m)-2-2\delta_2(m)
&
\delta_4(m)
\\
\mathrm E &
\tau(m)-1-\delta_2(m)
&
\delta_3(m)+\delta_4(m)+\delta_6(m)
\\
\mathrm F &
\tau(8m)-4
&
0
\\
\mathrm G & 1 & 0
\end{array}
\]
\end{theorem}

\begin{proof}
For Types~B--E, the acting unit groups induce the full unit actions on the intrinsic parameter groups: for Types~C and~E use the surjections \(U_{6m}\to U_m\), and for Type~D use \(U_{6m}\to U_{3m}\). Hence \cref{lem:orders} gives one unit orbit for each admissible additive order. The relevant bookkeeping from \cref{prop:parameters} is
\[
\begin{array}{c|c|c|c}
\text{Type}
& \text{parameter group}
& \text{removed order strata}
& \text{globally fixed admissible orders}
\\
\hline
\mathrm B
& \mathbb Z_{5m}
& 1,5,\ \text{and }10\text{ if }2\mid m
& 2,3,4,6
\\
\mathrm C
& \mathbb Z_m
& 1
& 2,3,4,6
\\
\mathrm D
& \mathbb Z_{3m}
& 1,3,\ \text{and }2,6\text{ if }2\mid m
& 4
\\
\mathrm E
& \mathbb Z_m
& 1,\ \text{and }2\text{ if }2\mid m
& 3,4,6
\end{array}
\]
Since the possible additive orders in \(\mathbb Z_R\) are precisely the divisors of \(R\), counting the remaining strata gives the Type~B--E orbit formulas. The last column and \cref{lem:orders} give their globally fixed counts.

The Type~A calculation is proved in \cref{app:A}. The Type~F formulas also follow directly from the orbit decomposition in \cref{thm:F-distribution}. For Type~G, the action in \cref{thm:action} is regular on four classes, so there is one orbit and no globally fixed class.
\end{proof}

\section{Stabilizers and orbit-size distributions}\label{sec:stabilizers}

The formulas in \cref{thm:counts} count unit orbits but do not describe their individual sizes. For Types~B--E and~G the orbit-size distributions follow directly from the preceding results, while Types~A and~F require a finer stabilizer analysis.

\subsection{Types B--E and G}

In this subsection, stabilizers are taken in the full unit group $U_N$ of the ambient bracelet modulus $N$. For Type~A below, we instead use the induced $U_m$-action, as stated explicitly there.

For Types~B--E, each admissible additive order \(q\) forms a single unit orbit by \cref{lem:orders}. Hence an orbit of order \(q\) has size
\[
\begin{cases}
1,&q=2,\\[0.4ex]
\varphi(q)/2,&q>2.
\end{cases}
\]
Equivalently, for an admissible parameter \(\alpha\) of additive order \(q\),
\[
|\operatorname{Stab}_{U_N}(\alpha)|
=
\begin{cases}
\varphi(N),&q=2,\\[0.6ex]
2\varphi(N)/\varphi(q),&q>2.
\end{cases}
\]
The admissible orders are those listed in \cref{prop:parameters}.

For Type~G, \cref{thm:action} gives one regular orbit of size \(4\); hence every Type~G class has stabilizer of size
\[
\frac{\varphi(20m)}4.
\]

\subsection{Type A: orders and relative phase}

Let \((\alpha,\beta)\) be an admissible Type~A parameter pair and put
\[
q=\ord(\alpha),
\qquad
r=\ord(\beta),
\qquad
g=\gcd(q,r),
\qquad
L=\operatorname{lcm}(q,r).
\]
Thus \(q,r\mid m\) and \(q,r>2\). Choose representatives
\[
\alpha=\left[\frac mq a\right]_\pm,\qquad a\in U_q,
\qquad
\beta=\left[\frac mr b\right]_\pm,\qquad b\in U_r.
\]
Reduction modulo \(g\) permits us to define the relative phase
\[
\lambda(\alpha,\beta)
=
[ab^{-1}]_\pm\in U_g/\{\pm1\}.
\]
Changing either representative by its negative changes \(ab^{-1}\) only by sign, so \(\lambda\) is well defined.

\begin{proposition}[Type A orbit classification]\label{prop:A-stabilizers}
Two admissible Type~A pairs lie in the same unit orbit if and only if they have the same ordered pair of additive orders \((q,r)\) and the same relative phase \(\lambda\).

For a pair of orders \(q,r\), the stabilizer in the induced \(U_m\)-action is
\[
H_A(q,r)
=
\left\{
u\in U_m:
u\equiv\pm1\pmod q,\quad
u\equiv\pm1\pmod r
\right\}.
\]
Consequently every orbit with order pair \((q,r)\) has size
\[
|\operatorname{Orb}(\alpha,\beta)|
=
\begin{cases}
\dfrac{\varphi(L)}4,&g\le2,\\[1.2ex]
\dfrac{\varphi(L)}2,&g>2.
\end{cases}
\]
For an admissible ordered order-pair, the number of such orbits is
\[
\nu_A(q,r)=
\begin{cases}
1,&q\ne r,\ g\le2,\\[0.6ex]
\dfrac{\varphi(g)}2,&q\ne r,\ g>2,\\[1.2ex]
\dfrac{\varphi(q)}2-1,&q=r.
\end{cases}
\]
As in \cref{prop:parameters}, the condition \(r=3\) forces \(q=6\).
\end{proposition}

\begin{proof}
A unit \(u\in U_m\) fixes a sign class of additive order \(q\) exactly when
\[
u\equiv\pm1\pmod q.
\]
Hence the displayed formula for \(H_A(q,r)\) follows immediately.

The four possible sign choices in the two congruences are simultaneously compatible exactly when the chosen signs agree modulo \(g\). The two equal-sign choices are always compatible, whereas the two opposite-sign choices are compatible exactly when \(g\mid2\). Since \(U_m\to U_L\) is surjective,
\[
|H_A(q,r)|
=
\frac{\varphi(m)}{\varphi(L)}
\begin{cases}
4,&g\le2,\\
2,&g>2.
\end{cases}
\]
The orbit-stabilizer theorem gives the stated orbit sizes.

The relative phase is invariant under the diagonal action, since multiplication of both \(a\) and \(b\) by the same unit cancels in \(ab^{-1}\). Conversely, suppose two pairs with the same orders are represented by \(a,b\) and \(a',b'\), and have the same relative phase. Then
\[
a'a^{-1}\equiv\pm b'b^{-1}\pmod g.
\]
After changing one target representative by sign if necessary, the congruences
\[
u\equiv a'a^{-1}\pmod q,
\qquad
u\equiv b'b^{-1}\pmod r
\]
are compatible. They determine a unit residue modulo \(L\), which lifts to a unit modulo \(m\). This unit carries the first parameter pair to the second.

Thus, when \(q\ne r\), the orbits with fixed \((q,r)\) are indexed by \(U_g/\{\pm1\}\), which has one element for \(g\le2\) and \(\varphi(g)/2\) elements for \(g>2\). When \(q=r\), the identity relative phase corresponds precisely to \(\alpha=\beta\), which is excluded, leaving \(\varphi(q)/2-1\) orbits.
\end{proof}

Thus the Type~A action has a simple interpretation: the additive orders of the two parameters and their relative phase on the common divisor \(\gcd(q,r)\) form a complete set of orbit invariants.

\subsection{Type F: affine orbit decomposition}

Let
\[
N=8m=2^sM,
\qquad
M\text{ odd},
\qquad
s\ge3,
\]
and let \([i]_\pm\) be an allowed Type~F parameter. Write
\[
q=\ord(i)=2^a r,
\qquad
r\text{ odd}.
\]

\begin{proposition}[Type F stabilizers]\label{prop:F-stabilizers}
The stabilizer of \([i]_\pm\) in \(U_N\) consists precisely of the units satisfying one of
\[
\begin{aligned}
&\varepsilon(u)=0,
&&u\equiv\pm1\pmod q,\\
&\varepsilon(u)=1,\quad q\text{ even},
&&u\equiv q/2\pm1\pmod q.
\end{aligned}
\]
Write this stabilizer as \(H_F(q)\). If \(L=\operatorname{lcm}(8,q)\), then
\[
|H_F(q)|
=
\frac{\varphi(N)}{\varphi(L)}
\begin{cases}
4,&0\le a\le3,\\
2,&a\ge4.
\end{cases}
\]
Consequently
\[
|\operatorname{Orb}([i]_\pm)|
=
\begin{cases}
\varphi(r),&0\le a\le3,\\[0.6ex]
2^{a-2}\varphi(r),&a\ge4.
\end{cases}
\]
\end{proposition}

\begin{proof}
If \(\varepsilon(u)=0\), fixedness means
\[
[ui]_\pm=[i]_\pm,
\]
which is equivalent to \(u\equiv\pm1\pmod q\).

Suppose \(\varepsilon(u)=1\). Then fixedness is equivalent to
\[
(u-1)i=N/2
\qquad\text{or}\qquad
(u+1)i=N/2.
\]
If \(q\) is odd, the subgroup generated by \(i\) contains no element of order \(2\), so neither congruence is possible. If \(q\) is even, then the unique order-two element of \(\langle i\rangle\) is
\[
(q/2)i=N/2,
\]
and the two conditions become
\[
u\equiv q/2+1\pmod q
\qquad\text{or}\qquad
u\equiv q/2-1\pmod q.
\]

It remains to count the compatible residue classes modulo \(L=\operatorname{lcm}(8,q)\). For \(a=0\) or \(1\), the \(\varepsilon=0\) case gives four classes and the \(\varepsilon=1\) case gives none. For \(a=2\) or \(3\), each case gives two classes. For \(a\ge4\), only the two \(\varepsilon=0\) classes remain, because \(q/2\pm1\equiv\pm1\pmod8\), whereas \(\varepsilon(u)=1\) requires \(u\equiv\pm3\pmod8\). Surjectivity of \(U_N\to U_L\) gives the stabilizer formula, and orbit-stabilizer gives the asserted orbit sizes.
\end{proof}

The affine term can change additive order only between the odd and twice-odd strata. More precisely,
\[
\ord(i+N/2)=
\begin{cases}
2q,&q\text{ odd},\\
q/2,&q\equiv2\pmod4,\\
q,&4\mid q.
\end{cases}
\]
Indeed, if \(q\) is odd, \(i\) and \(N/2\) have coprime orders \(q\) and \(2\). If \(q\) is even, then \(N/2=(q/2)i\), so
\[
i+N/2=(1+q/2)i,
\]
and the assertion follows from \(\ord(ki)=q/\gcd(q,k)\).

\begin{theorem}[Type F orbit-size distribution]\label{thm:F-distribution}
For every odd divisor \(r>1\) of \(M\):
\begin{enumerate}[label=\textup{(\roman*)}]
\item The sign classes of orders \(r\) and \(2r\), taken together, form one orbit of size \(\varphi(r)\).
\item The sign classes of order \(4r\) form one orbit of size \(\varphi(r)\).
\item The sign classes of order \(8r\) form exactly two orbits, each of size \(\varphi(r)\).
\end{enumerate}
Moreover, for every odd divisor \(r\mid M\) and every \(4\le a\le s\), the sign classes of order \(2^a r\) form one orbit of size
\[
2^{a-2}\varphi(r).
\]

For order \(8r\), the two orbits are distinguished as follows. Multiplication by \(r\) removes the odd-order component of \(i\), leaving an element of order \(8\). Thus there is a unique \(c\in U_8\) such that
\[
ri=cm.
\]
Replacing \(i\) by \(-i\) replaces \(c\) by \(-c\), so the sign class
\[
[c]_\pm\in U_8/\{\pm1\}
\]
is well defined. It is invariant under the Type~F action, and its two possible values \([\pm1]\) and \([\pm3]\) distinguish the two orbits.
\end{theorem}

\begin{proof}
For \(r>1\) odd, the total number of sign classes of orders \(r\) and \(2r\) is
\[
\frac{\varphi(r)}2+\frac{\varphi(2r)}2=\varphi(r).
\]
The affine part interchanges the two order strata, while \cref{prop:F-stabilizers} gives orbit size \(\varphi(r)\); hence they form a single orbit.

The number of order-\(4r\) sign classes is
\[
\frac{\varphi(4r)}2=\varphi(r),
\]
again equal to the orbit size, so this stratum is one orbit.

For order \(8r\), there are
\[
\frac{\varphi(8r)}2=2\varphi(r)
\]
sign classes, while each orbit has size \(\varphi(r)\), so there are exactly two orbits. To distinguish them, write \(ri=cm\) as above. Under the affine action,
\[
c\longmapsto uc+4\varepsilon(u)\pmod8.
\]
If \(\varepsilon(u)=0\), this is \(\pm c\). If \(\varepsilon(u)=1\), then \(u\equiv\pm3\pmod8\), and, since \(c\) is odd, \(uc+4\equiv\pm c\pmod8\). Thus \([c]_\pm\) is invariant, and the two sign classes in \(U_8/\{\pm1\}\) give the two orbits.

Finally, for \(a\ge4\), the number of sign classes of order \(2^a r\) is
\[
\frac{\varphi(2^a r)}2
=
2^{a-2}\varphi(r),
\]
which equals the orbit size in \cref{prop:F-stabilizers}. Hence the entire order stratum is one orbit.
\end{proof}

As a consequence, if \(D=\tau(M)\), then
\[
\begin{aligned}
O_F(m)
&=(D-1)+(D-1)+2(D-1)+(s-3)D\\
&=(s+1)D-4\\
&=\tau(N)-4\\
&=\tau(8m)-4.
\end{aligned}
\]
This recovers the Type~F orbit formula in \cref{thm:counts} directly from the orbit decomposition. It also gives \(F_F(m)=0\), since none of the allowed Type~F orbits has size one.

\subsection{Complementary classes}

Complementation preserves homometry for subsets of equal cardinality. Indeed, if $C_S(t)=|S\cap(S+t)|$, then $C_{S^c}(t)=N-2|S|+C_S(t)$. Moreover $T_c(S^c)=T_c(S)^c,\,R_c(S^c)=R_c(S)^c,\,u(S^c)=(uS)^c$. Thus complementation is equivariant for both the dihedral and unit actions. Consequently the unit orbits and stabilizers of the five-point families transport unchanged to their complementary $(N-5)$-point homometry classes.

\section{Concluding remarks}

The results above prove the unit-invariance conjecture for all seven families and determine their unit-orbit counts, globally fixed classes, and orbit-size distributions. In Type~A, the ordered additive orders together with a relative phase on their common divisor form a complete set of orbit invariants. In Type~F, the affine action has a rigid order-stratum decomposition: the odd and twice-odd strata pair, the order-\(8r\) strata split into two orbits, and all higher \(2\)-power strata are transitive.

Complementation transports these actions equivariantly to complementary homometry classes. It remains open whether homometry classes of larger cardinality not obtained by complementation admit comparably explicit arithmetic parameterizations and unit actions.

\appendix

\section{Dihedral verification of the parameter identities}
\label{app:param-identities}

Recall that $T_c(S)=S+c$ and $R_c(S)=c-S$ with all coordinates taken modulo \(N\). The maps \(T_c\) and \(R_c\) are respectively a translation and a reflection, so \([S]=[T_c(S)]=[R_c(S)]\). For each family, the parenthesized superscripts \( (1),(2) \), and, for Type~E, \( (3) \), are labels for the bracelets in the order in which they are displayed in \cref{sec:families}. Since a homometry class is an unordered set of bracelets, the following identities prove all of the parameter identifications used in \cref{prop:parameters}.

\smallskip
\noindent\textbf{Type A.}
In \(\mathbb Z_{2m}\),
\begin{align*}
T_m\!\left(A^{(1)}_{i+m,j}\right)&=A^{(1)}_{i,j},&
T_m\!\left(A^{(2)}_{i+m,j}\right)&=A^{(2)}_{i,j},\\
T_m\!\left(A^{(1)}_{i,j+m}\right)&=A^{(2)}_{i,j},&
T_m\!\left(A^{(2)}_{i,j+m}\right)&=A^{(1)}_{i,j},\\
R_m\!\left(A^{(1)}_{-i,j}\right)&=A^{(2)}_{i,j},&
R_m\!\left(A^{(2)}_{-i,j}\right)&=A^{(1)}_{i,j},\\
T_m\!\left(A^{(1)}_{i,-j}\right)&=A^{(2)}_{i,j},&
T_m\!\left(A^{(2)}_{i,-j}\right)&=A^{(1)}_{i,j}.
\end{align*}
Hence
\begin{align*}
\Acal_m(i+m,j)&=\Acal_m(i,j),&
\Acal_m(i,j+m)&=\Acal_m(i,j),\\
\Acal_m(-i,j)&=\Acal_m(i,j),&
\Acal_m(i,-j)&=\Acal_m(i,j).
\end{align*}

\smallskip
\noindent\textbf{Type B.}
In \(\mathbb Z_{5m}\),
\[
R_{2m}\!\left(B^{(1)}_{-i}\right)=B^{(1)}_i,
\qquad
R_m\!\left(B^{(2)}_{-i}\right)=B^{(2)}_i.
\]
Also \(B^{(k)}_{i+5m}=B^{(k)}_i\) for \(k=1,2\), directly modulo \(5m\).
Therefore
\[
\Bcal_m(-i)=\Bcal_m(i),
\qquad
\Bcal_m(i+5m)=\Bcal_m(i).
\]

\smallskip
\noindent\textbf{Type C.}
In \(\mathbb Z_{6m}\),
\begin{align*}
R_{4m}\!\left(C^{(1)}_{-i}\right)&=C^{(2)}_i,&
R_{4m}\!\left(C^{(2)}_{-i}\right)&=C^{(1)}_i,\\
C^{(1)}_{i+m}&=C^{(2)}_i,&
T_{4m}\!\left(C^{(2)}_{i+m}\right)&=C^{(1)}_i.
\end{align*}
Consequently
\[
\Ccal_m(-i)=\Ccal_m(i),
\qquad
\Ccal_m(i+m)=\Ccal_m(i).
\]

\smallskip
\noindent\textbf{Type D.}
In \(\mathbb Z_{6m}\),
\begin{align*}
R_{2m}\!\left(D^{(1)}_{-i}\right)&=D^{(2)}_i,&
R_{2m}\!\left(D^{(2)}_{-i}\right)&=D^{(1)}_i,\\
T_{3m}\!\left(D^{(1)}_{i+3m}\right)&=D^{(2)}_i,&
T_{3m}\!\left(D^{(2)}_{i+3m}\right)&=D^{(1)}_i.
\end{align*}
Thus
\[
\Dcal_m(-i)=\Dcal_m(i),
\qquad
\Dcal_m(i+3m)=\Dcal_m(i).
\]

\smallskip
\noindent\textbf{Type E.}
In \(\mathbb Z_{6m}\),
\begin{align*}
T_i\!\left(E^{(1)}_{-i}\right)&=E^{(2)}_i,&
T_i\!\left(E^{(2)}_{-i}\right)&=E^{(1)}_i,&
R_{2m}\!\left(E^{(3)}_{-i}\right)&=E^{(3)}_i,\\
R_{3m+i}\!\left(E^{(1)}_{i+m}\right)&=E^{(2)}_i,&
R_{2m}\!\left(E^{(2)}_{i+m}\right)&=E^{(3)}_i,&
T_i\!\left(E^{(3)}_{i+m}\right)&=E^{(1)}_i.
\end{align*}
Hence
\[
\Ecal_m(-i)=\Ecal_m(i),
\qquad
\Ecal_m(i+m)=\Ecal_m(i).
\]

\smallskip
\noindent\textbf{Type F.}
In \(\mathbb Z_{8m}\),
\[
T_i\!\left(F^{(1)}_{-i}\right)=F^{(2)}_i,
\qquad
T_i\!\left(F^{(2)}_{-i}\right)=F^{(1)}_i.
\]
Also \(F^{(k)}_{i+8m}=F^{(k)}_i\) for \(k=1,2\), directly modulo \(8m\). Therefore
\[
\Fcal_m(-i)=\Fcal_m(i),
\qquad
\Fcal_m(i+8m)=\Fcal_m(i).
\]

\section{Type A: evaluation of the Burnside sum}\label{app:A}

We derive the Type~A formulas in \cref{thm:counts}. Reduction \(U_{2m}\to U_m\) is surjective, and \(u\) and \(-u\) act identically on sign classes, so Burnside's lemma may be applied using \(U_m\).

Fix \(u\in U_m\), and set
\[
A_u=\gcd(m,u-1),
\qquad
B_u=\gcd(m,u+1),
\qquad
c_m=\gcd(m,2)=1+\delta_2(m).
\]
A sign class \([x]_\pm\) is fixed by \(u\) exactly when
\[
(u-1)x=0
\qquad\text{or}\qquad
(u+1)x=0
\]
in \(\mathbb Z_m\). The two solution sets have cardinalities \(A_u\) and \(B_u\), and their intersection has cardinality
\[
\gcd(m,u-1,u+1)=\gcd(m,2)=c_m.
\]
Indeed, every common divisor of \(u-1\) and \(u+1\) divides \(2\). If \(m\) is odd, both sides equal \(1\). If \(m\) is even, then \(u\in U_m\) is odd, so \(2\) divides each of \(m,u-1,u+1\), and both sides equal \(2\).

Removing \(0\), removing the order-two element when \(2\mid m\), and pairing every remaining element with its negative shows that the number of admissible individual Type~A sign-class parameters fixed by \(u\) is
\[
g_m(u)=\frac{A_u+B_u}{2}-c_m.
\]

The fixed Type~A pairs split according to the order of the second parameter. If \(\ord(\beta)\ne3\), there are
\[
\bigl(g_m(u)-\delta_3(m)\bigr)\bigl(g_m(u)-1\bigr)
\]
choices: the first factor chooses \(\beta\) outside the unique order-three sign class, when it exists, and the second then chooses \(\alpha\ne\beta\). If \(\ord(\beta)=3\), admissibility forces \(\alpha\) to be the unique order-six sign class. This contributes one further fixed pair when \(6\mid m\), namely \(\delta_6(m)\). Hence
\[
O_A(m)
=
\frac1{\varphi(m)}
\sum_{u\in U_m}
\left[
\bigl(g_m(u)-\delta_3(m)\bigr)\bigl(g_m(u)-1\bigr)
+\delta_6(m)
\right].
\]

It remains to evaluate this average. We use the three moments
\[
\frac1{\varphi(m)}\sum_{u\in U_m}A_u=\tau(m),
\qquad
\frac1{\varphi(m)}\sum_{u\in U_m}A_u^2=P(m),
\qquad
\frac1{\varphi(m)}\sum_{u\in U_m}A_uB_u=\kappa(m).
\]

For the first identity, use
\[
\gcd(m,u-1)
=
\sum_{\substack{d\mid m\\u\equiv1\;(\mathrm{mod}\ d)}}\varphi(d).
\]
For each \(d\mid m\), surjectivity of \(U_m\to U_d\) shows that exactly \(\varphi(m)/\varphi(d)\) units satisfy \(u\equiv1\pmod d\). Averaging therefore gives one contribution for each divisor \(d\mid m\), namely \(\tau(m)\).

For the second identity, let \(J_2\) be the second Jordan totient function. Since
\[
n^2=\sum_{d\mid n}J_2(d),
\]
we obtain
\[
A_u^2
=
\sum_{\substack{d\mid m\\u\equiv1\;(\mathrm{mod}\ d)}}J_2(d).
\]
Averaging as above gives
\[
\frac1{\varphi(m)}\sum_{u\in U_m}A_u^2
=
\sum_{d\mid m}\frac{J_2(d)}{\varphi(d)}
=
\sum_{d\mid m}\psi(d)
=
P(m),
\]
because \(J_2(d)/\varphi(d)=\psi(d)\).

For the mixed moment, expanding both gcds gives
\[
\frac1{\varphi(m)}\sum_{u\in U_m}A_uB_u
=
\sum_{\substack{d,e\mid m\\ \gcd(d,e)\mid2}}
\frac{\varphi(d)\varphi(e)}{\varphi(\operatorname{lcm}(d,e))}.
\]
Indeed, the simultaneous congruences
\[
u\equiv1\pmod d,
\qquad
u\equiv-1\pmod e
\]
are compatible exactly when \(\gcd(d,e)\mid2\); when compatible, they determine a unit residue modulo \(\operatorname{lcm}(d,e)\), and surjectivity of the reduction map counts its lifts.

The last divisor sum is multiplicative in \(m\). For an odd prime power \(p^a\), compatibility forces at least one of the two local divisor exponents to be zero, so the local factor is
\[
1+2a=2a+1.
\]
For \(2^a\) with \(a\ge1\), the allowed exponent pairs \((i,j)\) satisfy \(\min(i,j)\le1\). The pairs with \(i=0\) or \(j=0\) contribute \(2a+1\), while the additional pairs with \(i=1\) or \(j=1\) contribute \(2a-1\). Thus the local factor is \(4a\). This is exactly the multiplicative function \(\kappa(m)\) defined in \cref{sec:orbits}.

Since \(u\mapsto-u\) interchanges \(A_u\) and \(B_u\), the same first and second moments hold for \(B_u\). Therefore
\[
\frac1{\varphi(m)}\sum_{u\in U_m}g_m(u)
=
\tau(m)-c_m
\]
and
\[
\frac1{\varphi(m)}\sum_{u\in U_m}g_m(u)^2
=
\frac{P(m)+\kappa(m)}2
-2c_m\tau(m)+c_m^2.
\]
Expanding the Burnside expression and substituting \(c_m=1+\delta_2(m)\) yields
\[
\boxed{
O_A(m)=
\frac{P(m)+\kappa(m)}2
-\bigl(3+2\delta_2(m)+\delta_3(m)\bigr)\tau(m)
+2+4\delta_2(m)+2\delta_3(m)+2\delta_6(m).
}
\]

By \cref{lem:orders}, nonzero sign classes fixed by every unit have orders $2,3,4,$ or $6$. Since Type~A excludes order $2$, the admissible ones have orders $3,4,$ or $6$. Thus
\[
s_m=\delta_3(m)+\delta_4(m)+\delta_6(m),
\]
and the same pair count gives
\[
F_A(m)=
\bigl(s_m-\delta_3(m)\bigr)(s_m-1)+\delta_6(m).
\]
If \(6\nmid m\), this is \(0\); if \(6\mid m\) but \(12\nmid m\), it is \(2\); and if \(12\mid m\), it is \(5\). Equivalently,
\[
\boxed{F_A(m)=2\delta_6(m)+3\delta_{12}(m).}
\]

\end{document}